\documentclass[12pt]{article}
\usepackage{color}
\usepackage{graphicx}
\usepackage{colortbl}
\usepackage{array}
\usepackage{amssymb}
\usepackage{amsmath}
\usepackage{amsthm}
\usepackage{mathrsfs}
\usepackage{dcolumn}
\usepackage{longtable}
\usepackage{hhline}
\numberwithin{equation}{section}
\allowdisplaybreaks

\newtheorem{thm}{Theorem}[section]
\newtheorem{defn}[thm]{Definition}

\title{Translated Logarithmic Lambert Function and its Applications to Three-Parameter Entropy}

\author{\textbf{Cristina B. Corcino$^{1,2}$}\\ {\large\bf Roberto B. Corcino$^{1,2}$}\\ {$^1$Research Institute for Computational}\\ {Mathematics and Physics}\\{$^2$Department of Mathematics}\\Cebu Normal University\\Cebu City, Philippines \vspace{9pt} 
}

\begin{document}

\maketitle

\begin{abstract}
The translated logarithmic Lambert function is defined and basic analytic properties of the function are obtained including the derivative, integral, Taylor series expansion, real branches and asymptotic approximation of the function. Moreover, the probability distribution of the three-parameter entropy is derived which is expressed in terms of the translated logarithmic Lambert function.



\end{abstract}

\bigskip

\section{Introduction}
The first definition of entropy that appeared in the literature is in the context of thermodynamics. Being commonly understood as a measure of disorder, entropy is defined in  thermodynamics viewpoint as a measure of the number of specific ways in which a thermodynamic system may be arranged. However, the microscopic details of a system are not considered in this context. The definition of entropy in the statistical mechanics point of view appeared later along with other thermodynamic properties. In this context, entropy is considered as an extensive property of a thermodynamic system wherein thermodynamic properties are defined in terms of the statistics of the motions of the microscopic constituents of a system.

\smallskip
It is known that the entropy of an isolated system never decreases, which is the essence of the second law of thermodynamics. Such a system will spontaneously proceed towards thermodynamic equilibrium, the configuration with maximum entropy \cite{Gibbs}. There are three macroscopic variables that describe a system in thermodynamic equilibrium which correspond to thermal, mechanical and the chemical equilibrium. To each value of these macroscopic variables, there exist several possible microscopic configurations. These will then entail different systems and the collection of these systems is called an ensemble. One of the popular ensembles is the canonical ensemble, which 
is statistical in nature that represents the possible states of a mechanical system in thermal equilibrium with a heat bath at a fixed temperature.

\smallskip
Some of the physical systems cannot be described by Boltzmann-Gibbs(BG) statistical mechanics \cite{Asgarani, 151569, Shlesinger-Zaslavsky-Klafter, Bediaga-Curado-deMiranda, Walton-Rafelski, Binney-Tremaine, Clayton}. However, Tsallis \cite{Tsallis1988} has overcome some of these difficulties by introducing the following $q$-entropy
\begin{equation}
S_q=k\sum_{i=1}^{\omega}p_i\ln_q\frac{1}{p_i},
\end{equation}
where $k$ is a positive constant and $\omega$ is the total number of microscopic states. For any real number $x$ and $q>0$, $\ln_q x$ called the $q$-logarithm is defined by
\begin{equation}
\ln_q x=\frac{x^{1-q}-1}{1-q},\;\;\;\ln_1x=\ln x.
\end{equation}
The inverse function of the $q$-logarithm is called $q$-exponential and is given by
\begin{equation}
\exp_q x=[1+(1-q)x]^{\frac{1}{1-q}},\;\;\;\exp_1x=\exp x.
\end{equation}
In the case of equiprobability, BG is recovered in the limit $q\to1$. 

\smallskip
A two-parameter entropy $S_{q,q'}$ that recovered the $q$-entropy $S_q$ in the limit $q'\to1$ was defined in \cite{Schwammle-Tsallis} as
\begin{equation}\label{twoparaentropy}
S_{q,q'}\equiv\sum_{i=1}^{\omega}p_i\ln_{q,q'}\frac{1}{p_i}=\frac{1}{1-q'}\sum_{i=1}^{\omega}p_i\left[\exp\left(\frac{1-q'}{1-q}(p_i^{q-1}-1)\right)-1\right].
\end{equation}
Applications of $S_q$ to a class of energy based ensembles were done in \cite{Chandrashekar-Mohammed} while applications of $S_{q,q'}$ to adiabatic ensembles were done in \cite{Chandrashekar-Segar}. Results in the applications of $S_{q,q'}$ involved the well-known Lambert W function.

\smallskip
A three-parameter entropy $S_{q,q',r}$ that recovers $Sq,q'$ in the limit $r\to1$ was defined in \cite{Corcino-Corcino} as
\begin{equation}\label{three-parameter entropy}
S_{q,q^{\prime },r} \equiv k\sum_{i=1}^w{p_i\ln_{ q,q^{\prime },r} \frac{1}{p_i}},
\end{equation}
where k is a positive constant and 
\begin{equation}
\label{lnq}
\ln_{q,q',r}x \equiv \frac{1}{1-r}\left(\exp\left(\frac{1-r}{1-q'}\left(e^{(1-q')\ln_q x}-1\right)-1\right)\right)
\end{equation}
The three-parameter entropic function \eqref{three-parameter entropy} was shown to be analytic (hence, Lesche-stable), concave and convex in specified ranges of the parameters (see \cite{Corcino-Corcino}). 

\smallskip
In this paper another variation of Lambert W function called the translated logarithmic Lambert function will be introduced. Moreover, the probability distribution of the three-parameter entropy is derived and expressed in terms of the translated logarithmic Lambert function.

\section{Translated Logarithmic Lambert Function}

The generalization of the Lambert $W$ function introduced here is the translated logarithmic Lambert function denoted by $W_\mathcal{LT}(x)$ and is defined as follows: 

\begin{defn}\label{def1}\rm
For any real number $x$ and constant $B$, the translated logarithmic Lambert function $W_\mathcal{LT}(x)$ is defined to be the solution to the equation
\begin{equation}
(Ay\ln(By)+y+C)e^y=x.
\label{defn of log lambert}
\end{equation}
\end{defn}

\smallskip
Observe that $y$ cannot be zero. Moreover, $By$ must be positive. By Definition \ref{def1}, $y=W_\mathcal{LT}(x)$. The derivatives of $W_\mathcal{LT}(x)$ with respect to $x$ can be readily determined as the following theorem shows. 

\begin{thm} The derivative of the translated logarithmic Lambert function is given by
\begin{equation}
\frac{dW_\mathcal{LT}(x)}{dx}=\frac{e^{-W_\mathcal{LT}(x)}}{[W_\mathcal{LT}(x)+1]A\ln BW_\mathcal{LT}(x) +W_\mathcal{LT}(x)+A+C+1}.
\label{derivatives of log lambert}
\end{equation}
\end{thm}
\begin{proof} Taking the derivative of both sides of \eqref{defn of log lambert} gives 
\begin{align*}
(Ay\ln(By) +y+C)e^y \frac{dy}{dx}+\left(A+A\ln(By)+1\right)e^y \ \frac{dy}{dx}=1, 
\end{align*}
from which
\begin{equation}
\frac{dy}{dx}=\frac{1}{\left[Ay\ln(By)+y+C+A+A\ln(By)+1\right]e^y}.
\label{dy/dx}
\end{equation}
With $y=W_\mathcal{LT}(x)$, \eqref{dy/dx} reduces to \eqref{derivatives of log lambert}.
\end{proof}

\smallskip
The integral of the translated logarithmic Lambert function is given in the next theorem.
\begin{thm} The integral of $W_\mathcal{LT}(x)$ is
\begin{align}
\int W_\mathcal{LT}(x)\ dx &=e^{W_\mathcal{LT}(x)}\left[\left(W^2_\mathcal{LT}(x)-W_\mathcal{LT}(x)+1\right)A\ln\left(BW_\mathcal{LT}(x)\right)+W^2_\mathcal{LT}(x)\right.\nonumber\\
&\left. +(C-1)W_\mathcal{LT}(x)+1+A-C\right]-2Ei\left(W_\mathcal{LT}(x)\right)+C',
\label{integral of log lambert}
\end{align}
where $Ei(x)$ is the exponential integral given by
$$Ei(x)=\int \frac{e^x}{x}dx.$$
\end{thm}
\begin{proof} From \eqref{defn of log lambert},
\begin{equation*}
dx=\left[Ay\ln(By)+y+C+A+A\ln(By)+1\right]e^y\ dy.
\end{equation*}
Thus,
\begin{align}
\int y\ dx&=\int y\left[Ay\ln(By)+y+C+A+A\ln(By)+1\right]e^y\ dy \nonumber \\
&=A\int y^2e^y\ln(By)\ dy+A\int ye^y\ln(By)\ dy+\int y^2e^y\ dy \nonumber \\
&\;\;\;\;\;+(C+A+1)\int ye^ydy. 
\label{integral of y dx}
\end{align}
These integrals can be computed using integration by parts to obtain
\begin{equation}
(A+C+1)\int ye^y\ dy=(A+C+1)(y-1)e^y+C_1,
\label{fourth integral}
\end{equation}
\begin{equation}
\int y^2e^y\ dy=(y^2-2y+2)e^y+C_2,
\label{third integral}
\end{equation}
\begin{equation}
A\int ye^y\ln(By)\ dy=Ae^y\left((y-1)\ln(By)-1\right)+Ei(y)+C_3,
\label{second integral}
\end{equation}
\begin{equation}
A\int y^2e^y\ln(By)\ dy=Ae^y\left[(y^2-2y+2)\ln(By)-y+3\right]-2Ei(y)+C_4,
\label{first integral}
\end{equation}
where $C_1, C_2, C_3$ are constants. Substitution of \eqref{fourth integral}, \eqref{third integral}, \eqref{first integral} and \eqref{second integral} to \eqref{integral of y dx} with $C'=C_1+C_2+C_3+C_4$, and writing $W_\mathcal{LT}(x)$ for $y$  will give \eqref{integral of log lambert}.
\end{proof}

\smallskip
The next theorem contains the Taylor series expansion of $W_\mathcal{LT}(x)$.

\begin{thm} Few terms of the Taylor series of $W_\mathcal{LT}(x)$ about 0 are given below:
\begin{equation}
W_\mathcal{LT}(x)=\frac{1}{B}e^{W\left(\frac{-BCe^{1/A}}{A}\right)-\frac{1}{A}}+\frac{e^{\frac{-1}{B}e^{W\left(\frac{-BCe^{1/A}}{A}\right)-\frac{1}{A}}}}{A\left[W\left(\frac{-BCe^{1/A}}{A}\right)+1\right]} \ x+\cdots
\label{Taylor series}
\end{equation}
where $W(x)$ is the classical Lambert W function.
\end{thm}
\begin{proof} Being the inverse of the function defined by $x=y\ln(By)e^y$, the Lagrange inversion theorem is the key to obtain the Taylor series of the function $W_\mathcal{LT}(x)$. \\
\indent Let $f(y)=(Ay\ln(By)+y+C)e^y$. The function $f$ is analytic for $By>0$. Moreover, 
$f'(y)=\left[A(y+1)\ln(By)+y+A+C+1\right]e^y$, 
$$f'\left(\frac{1}{B}e^{W\left(\frac{-BCe^{1/A}}{A}\right)-\frac{1}{A}}\right)=Ae^{\frac{1}{B}e^{W\left(\frac{-BCe^{1/A}}{A}\right)-\frac{1}{A}}}\left[W\left(\frac{-BCe^{1/A}}{A}\right)+1\right]\neq 0, $$ 
where  $W\left(\frac{-BCe^{1/A}}{A}\right)\neq -1 \ (A\neq0)$, and for finite $B$, 
$$f\left(\frac{1}{B}e^{W\left(\frac{-BCe^{1/A}}{A}\right)-\frac{1}{A}}\right)=0.$$ 
By the Lagrange Inversion Theorem, taking $a=\frac{1}{B}e^{W\left(\frac{-BCe^{1/A}}{A}\right)-\frac{1}{A}}$, we have
\begin{equation}
W_\mathcal{LT}(x)=\frac{1}{B}e^{W\left(\frac{-BCe^{1/A}}{A}\right)-\frac{1}{A}}+\sum_{n = 1}^\infty g_n \frac{x^n}{n!},
\label{from Lagrange theorem}
\end{equation}
where
\begin{equation}
g_n=\lim_{y\to a} \frac{d^{n-1}}{dy^{n-1}} \left(\frac{y-\frac{1}{B}e^{W\left(\frac{-BCe^{1/A}}{A}\right)-\frac{1}{A}}}{f(y)}\right)^n.
\label{gn}
\end{equation}
That is, when $n=1$,
\begin{align*}
g_1=\frac{e^{\frac{-1}{B}e^{W\left(\frac{-BCe^{1/A}}{A}\right)-\frac{1}{A}}}}{A\left[W\left(\frac{-BCe^{1/A}}{A}\right)+1\right]}.
\end{align*}
Substituting to \eqref{from Lagrange theorem} will yield \eqref{Taylor series}.
\end{proof}


\smallskip
An approximation formula for $W_\mathcal{LT}(x)$ expressed in terms of the  classical Lambert $W$ function is proved in the next theorem.

\begin{thm} For large $x$,
\begin{align}
W_\mathcal{LT}(x) &\sim W\left(\frac{xe^{\frac{C}{A+1}}}{A+1}\right)-\ln\left\{\left(\frac{e^{\frac{C}{A+1}}}{A+1}\right)\left[A\ln\left(BW\left(\frac{xe^{\frac{C}{A+1}}}{A+1}\right)\right)+1\right]\right.\nonumber\\
&\;\;\;\;\;\;\left.+\frac{C}{x}e^{W\left(\frac{xe^{\frac{C}{A+1}}}{A+1}\right)}\right\}-\frac{C}{A+1}.\label{approximation}
\end{align}
where $W(x)$ denotes the Lambert $W$ function.
\end{thm}
\begin{proof} From \eqref{defn of log lambert}, $y=W_\mathcal{LT}(x)$ satisfies 
\begin{equation*}
x=[Ay(\ln By)+y+C]e^y\sim [(A+1)y+C]e^y.
\end{equation*}
Then
\begin{align}
y&=W\left(\frac{xe^{\frac{C}{A+1}}}{A+1}\right)-\frac{C}{A+1}+u(x)\label{eqn15}\\
&=W\left(\frac{xe^{\frac{C}{A+1}}}{A+1}\right)\left[1-\frac{\frac{C}{A+1}}{W\left(\frac{xe^{\frac{C}{A+1}}}{A+1}\right)}+\frac{u(x)}{W\left(\frac{xe^{\frac{C}{A+1}}}{A+1}\right)}\right],\label{expression with u(x)}
\end{align}
where $u(x)$ is a function to be determined. Substituting \eqref{expression with u(x)} to \eqref{defn of log lambert} yields
\begin{align}
&\left\{AW\left(\frac{xe^{\frac{C}{A+1}}}{A+1}\right)\left[1-\frac{\frac{C}{A+1}}{W\left(\frac{xe^{\frac{C}{A+1}}}{A+1}\right)}+\frac{u(x)}{W\left(\frac{xe^{\frac{C}{A+1}}}{A+1}\right)}\right]\times\right.\nonumber\\
&\;\;\;\left.\ln\left(BW\left(\frac{xe^{\frac{C}{A+1}}}{A+1}\right)\left[1-\frac{\frac{C}{A+1}}{W\left(\frac{xe^{\frac{C}{A+1}}}{A+1}\right)}+\frac{u(x)}{W\left(\frac{xe^{\frac{C}{A+1}}}{A+1}\right)}\right]\right)+\right.\nonumber\\
&\;\;\;\left.W\left(\frac{xe^{\frac{C}{A+1}}}{A+1}\right)\left[1-\frac{\frac{C}{A+1}}{W\left(\frac{xe^{\frac{C}{A+1}}}{A+1}\right)}+\frac{u(x)}{W\left(\frac{xe^{\frac{C}{A+1}}}{A+1}\right)}\right]+C\right\}\times\nonumber\\
&\;\;\;e^{W\left(\frac{xe^{\frac{C}{A+1}}}{A+1}\right)}\cdot e^{u(x)}=x.
\label{expression1 for x}
\end{align}
With $\frac{C}{A+1}, u(x)<<W\left(\frac{xe^{\frac{C}{A+1}}}{A+1}\right)$, \eqref{expression1 for x} becomes 
\begin{align*}
&\left\{AW\left(\frac{xe^{\frac{C}{A+1}}}{A+1}\right)e^{W\left(\frac{xe^{\frac{C}{A+1}}}{A+1}\right)}\ln\left(BW\left(\frac{xe^{\frac{C}{A+1}}}{A+1}\right)\right)+\right.\nonumber\\
&\;\;\;\;\;\;\;\;\;\;\;\;\;\;\;\left.\left[W\left(\frac{xe^{\frac{C}{A+1}}}{A+1}\right)+C\right]e^{W\left(\frac{xe^{\frac{C}{A+1}}}{A+1}\right)}\right\}e^{u(x)}=x.
\end{align*}
\begin{align*}
\left\{\left(\frac{xe^{\frac{C}{A+1}}}{A+1}\right)\left[A\ln\left(BW\left(\frac{xe^{\frac{C}{A+1}}}{A+1}\right)+1\right]+Ce^{W\left(\frac{xe^{\frac{C}{A+1}}}{A+1}\right)}\right]\right\}e^{u(x)}=x.
\end{align*}
\begin{align*}
\left\{\left(\frac{e^{\frac{C}{A+1}}}{A+1}\right)\left[A\ln\left(BW\left(\frac{xe^{\frac{C}{A+1}}}{A+1}\right)+1\right]+\frac{C}{x}e^{W\left(\frac{xe^{\frac{C}{A+1}}}{A+1}\right)}\right]\right\}e^{u(x)}=1.
\end{align*}
Thus,
\begin{align*}
u(x)=-\ln\left\{\left(\frac{e^{\frac{C}{A+1}}}{A+1}\right)\left[A\ln\left(BW\left(\frac{xe^{\frac{C}{A+1}}}{A+1}\right)\right)+1\right]+\frac{C}{x}e^{W\left(\frac{xe^{\frac{C}{A+1}}}{A+1}\right)}\right\}.
\end{align*}
Substituting this to \eqref{eqn15} yields \eqref{approximation}.
\end{proof}

\smallskip
The table below illustrates the accuracy of the approximation formula in \eqref{approximation} with $A=B=C=1$.

\begin{table*}[htbp]
	\centering
		\begin{tabular}{|c|c|c|c|}
		\hline
		$x$ & $W_\mathcal{LT}(x)$	&  Approximate       & Relative Error\\
		    &                     &  Value             &\\
		\hline
	$3575.7472$	& $4$ & $3.3121$ & $1.71987\times 10^{-1}$\\
		\hline
	$2084.7878$	& $5$ & $4.3301$ & $1.33982\times 10^{-1}$\\
		\hline
	$7161.0857$	& $6$ & $5.3453$ & $1.09116\times 10^{-1}$\\
	  \hline
	 $23710.7124$ & $7$ & $6.3581$ & $9.16961\times 10^{-2}$\\
	  \hline
	 $76418.4449$ & $8$ & $7.3690$ & $7.88738\times 10^{-2}$\\
	  \hline
	 $241269.4957$ & $9$ & $8.3783$ & $6.90741\times 10^{-2}$ \\
	  \hline
	 $749469.2416$ & $10$ & $9.3864$ & $6.13602\times 10^{-2}$\\
	 \hline
		\end{tabular}
\end{table*}
 
\bigskip
The next theorem describes the branches of the translated logarithmic Lambert function.  

\bigskip
\begin{thm}
Let $x=f(y)=[Ay\ln(By)+y+C]e^y$. Then the branches of the translated logarithmic Lambert function $y=W_\mathcal{LT}(x)$ can be described as follows:
\begin{enumerate}
\item When $B>0, A>0$, the branches are
\begin{itemize}
\item $W^0_\mathcal{LT}(x) : [f(\delta),f(0))\to (0, \delta]$ is strictly decreasing;
\item[]
\item $W^1_\mathcal{LT}(x) : [f(\delta), +\infty)\to [\delta, +\infty)$ is strictly increasing,
\end{itemize}
\item[]
\item When $B>0, A<0$, the branches are
\begin{itemize}
\item $W^0_\mathcal{LT}(x) : [f(0),f(\delta))\to (0, \delta]$ is strictly increasing;
\item[]
\item $W^1_\mathcal{LT}(x) : (-\infty, f(\delta)]\to [\delta, +\infty)$ is strictly decreasing,
\end{itemize}
where $\delta$ is the unique solution to
\begin{equation}\label{singu}
Ay\ln(By)+y+C+A+A\ln(By)=-1.
\end{equation}
\item When $B<0, A>0, |C|\leq A$, the branches are
\begin{itemize}
\item $W^0_{\mathcal{LT},<}(x) : (f(0), f(\delta_1)]\to [\delta_1,0)$ is strictly decreasing;
\item[]
\item $W^1_{\mathcal{LT},<}(x) : [f(\delta_2),f(\delta_1)]\to [\delta_2,\delta_1]$ is strictly increasing,
\item[]
\item $W^2_{\mathcal{LT},<}(x) : [f(\delta_2),0)\to (-\infty,\delta_2]$ is strictly decreasing,
\end{itemize}
\item[]
\item When $B<0, A<0, C\leq |A|$, the branches are
\begin{itemize}
\item $W^0_{\mathcal{LT},<}(x) : [f(\delta_1), f(0)]\to [\delta_1,0)$ is strictly increasing;
\item[]
\item $W^1_{\mathcal{LT},<}(x) : [f(\delta_1),f(\delta_2)]\to [\delta_2,\delta_1]$ is strictly decreasing,
\item[]
\item $W^2_{\mathcal{LT},<}(x) : (0, f(\delta_2)]\to (-\infty,\delta_2]$ is strictly increasing,
\end{itemize}
where $\delta_1$ and $\delta_2$ are the two solutions to \eqref{singu} with 
$$\delta_2<\frac{1}{B}e^{W\left(\frac{-BCe^{1/A}}{A}\right)-\frac{1}{A}}<\delta_1<0.$$ 
\end{enumerate}
\end{thm}
\begin{proof} Consider the case when $B>0, A>0$. Let $x=f(y)=[Ay\ln(By)+y+C]e^y$. From equation \eqref{derivatives of log lambert},  the derivative of $y=W_\mathcal{LT}(x)$ is not defined when $y$ satisfies \eqref{singu}. The solution $y=\delta$ to \eqref{singu} can be viewed as the intersection of the functions
$$g(y)=\frac{-y-C-A-1}{y+1}\;\;\;\;\mbox{and}\;\;\;\;h(y)=A\ln (By).$$
Clearly, the solution is unique. Thus, the derivative $\frac{dW_\mathcal{LT}(x)}{dx}$ is not defined for $x=f(\delta)=[A\delta\ln(B\delta)+\delta+C]e^{\delta}$. The value of $f(\delta)$ can then be used to determine the branches of $W_\mathcal{LT}(x)$. To explicitly identify the said branches, the following information are important:
\begin{enumerate}
\item the value of $y$ must always be positive, otherwise, $\ln (By)$ is undefined;
\item the function $y=W_\mathcal{LT}(x)$ has only one $y$-intercept, i.e., $y=\frac{1}{B}$;
\item if $y<\delta$, $A(y+1)\ln (By)+y+A+C+1>0$ which gives $\frac{dy}{dx}>0$;
\item if $y>\delta$, $A(y+1)\ln (By)+y+A+C+1<0$ which gives $\frac{dy}{dx}<0$;
\item if $y=\delta$, $A(y+1)\ln (By)+y+A+C+1=0$ and $\frac{dy}{dx}$ does not exist
\end{enumerate}
These imply that 
\begin{enumerate}
\item when $y<\delta$, the function $y=W_\mathcal{LT}(x)$ is increasing in the domain $[f(\delta),0)$ with range $(0, \delta]$ and the function crosses the $y$-axis only at 
$$y=\frac{1}{B}e^{W\left(\frac{-BCe^{1/A}}{A}\right)-\frac{1}{A}};$$
\item when $y>\delta$, the function $y=W_\mathcal{LT}(x)$ is decreasing, the domain is $[f(\delta),+\infty)$ and the range is $[\delta,+\infty)$ because this part of the graph does not cross the $x$-axis and $y$-axis;
\item when $y=\delta$, the line tangent to the curve at the point $(f(\delta),\delta)$ is a vertical line.
\end{enumerate}
These proved the case when $B>0, A>0$. The case where $B>0, A<0$ can be proved similarly. For the case $B<0, A>0, |C|\leq A$, the solution to \eqref{singu} can be viewed as the intersection of the functions
$$g(y)=\frac{-y-C-A-1}{y+1}\;\;\;\;\mbox{and}\;\;\;\;h(y)=A\ln (By).$$
These graphs intersect at two points $\delta_1$ and $\delta_2$. Thus, the derivative $\frac{dW_\mathcal{LT}(x)}{dx}$ is not defined for 
\begin{align*}
x_1&=f(\delta_1)=[A\delta_1\ln(B\delta_1)+\delta_1+C]e^{\delta_1},\\
x_2&=f(\delta_2)=[A\delta_2\ln(B\delta_2)+\delta_2+C]e^{\delta_2}.
\end{align*}
Note that
\begin{enumerate}
\item the value of $y$ must always be negative, otherwise, $\ln (By)$ is undefined;
\item the function $y=W_\mathcal{LT}(x)$ has only one $y$-intercept, i.e., 
$$y=\frac{1}{B}e^{W\left(\frac{-BCe^{1/A}}{A}\right)-\frac{1}{A}};$$
\item $g(y)$ is not defined at $y=-1$.
\end{enumerate}
The desired branches are completely determined as follows:
\begin{enumerate}
\item If $\delta_1<y<0$, then $A(y+1)\ln (By)+y+A+C+1<0$. This gives $\frac{dy}{dx}<0$. Thus, the function $y=W_\mathcal{LT}(x)$ is a decreasing function with domain $[f(0),f(\delta_1)]$ with range $[\delta_1,0]$;
\item If $\delta_2\le y\le \delta_1$, then $A(y+1)\ln (By)+y+A+C+1>0$. This gives $\frac{dy}{dx}>0$. Thus, the function $y=W_\mathcal{LT}(x)$ is increasing function with domain $[f(\delta_2),f(\delta_1)]$ and range $[\delta_2,\delta_1]$;
\item If $-\infty<y<\delta_2$, then $A(y+1)\ln (By)+y+A+C+1<0$. This gives $\frac{dy}{dx}<0$. Thus $y=W_\mathcal{LT}(x)$ is a decreasing function with domain $[f(\delta_2),0)$ and range $(-\infty,\delta_2]$.
\end{enumerate}
The case where $B<0, A<0, C\leq |A|$ can be proved similarly.
\end{proof}

Figures 1 and 2 depict the graphs of the translated logarithmic Lambert function (red color graphs) when $B=1$ and $B=-1$. The $y$-coordinates of the points of intersection of the blue and black colored graphs correspond to the value of $\delta, \delta_1$ and $\delta_2$.

\begin{figure}[t!]
\centerline{\includegraphics[width=7cm]{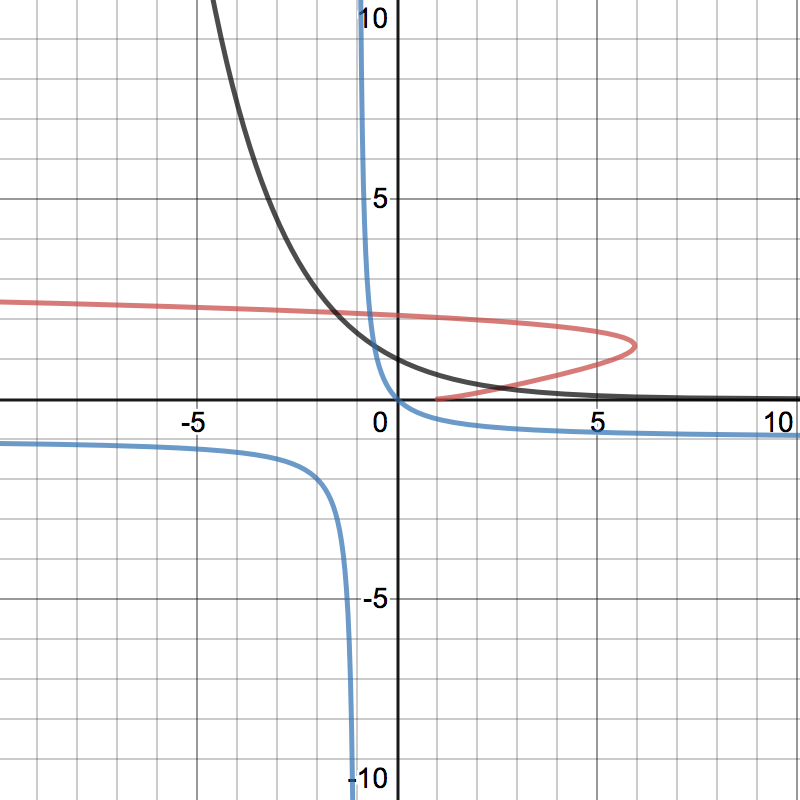}}
\centerline{\footnotesize{\textbf{Figure 1}}. {\footnotesize{Graph of translated logarithmic Lambert Function with $B = 1, A = 2, C=1$. }}}
\centerline{{\footnotesize{The graphs with red, blue and black colors are the graphs of }}}
\centerline{{\footnotesize{$x=f(y)$, $x=g(y)$ and $x=h(y)$, respectively.}}}
\end{figure}
\begin{figure}[hbt!]
\centerline{\includegraphics[width=7cm]{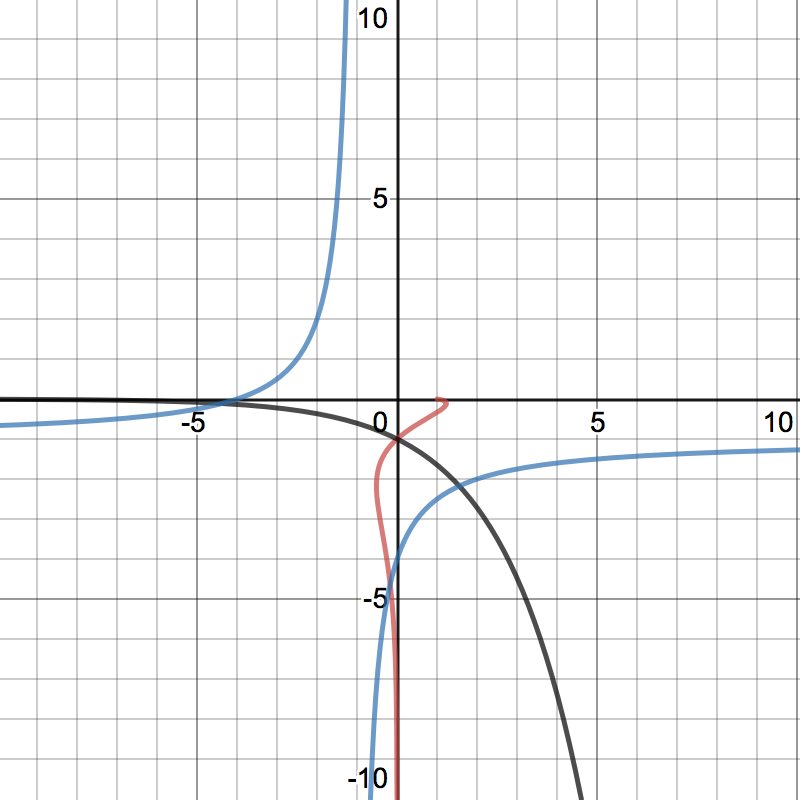}}
\centerline{\footnotesize{\textbf{Figure 2}}. {\footnotesize{Graph of translated logarithmic Lambert Function with $B = -1, A = -2, C=1$. }}}
\centerline{{\footnotesize{The graphs with red, blue and black colors are the graphs of }}}
\centerline{{\footnotesize{$x=f(y)$, $x=g(y)$ and $x=h(y)$, respectively.}}}
\end{figure}

\vspace{5pt}

\bigskip
\section{Applications to Entropy}

In this section, application of the translated logarithmic Lambert function to entropy in canonical ensemble is derived. Parallel to the two-parameter entropy in \eqref{twoparaentropy}, the three-parameter entropy, denoted by $S_{q,q',r}$, can also be constructed based on the three-parameter logarithm as follows:
\begin{align}
S_{q,q',r}&=k\sum_{i=1}^{\omega}p_i\ln_{q,q',r} \frac{1}{p_i}\\
&=k\sum_{i=1}^{\omega}p_i\frac{1}{1-r}\left(\exp\left(\frac{1-r}{1-q'}\left(e^{(1-q')\ln_q x}-1\right)-1\right)\right)
\end{align}
where $x=\frac{1}{p_i}$. In maximizing $S_{q,q',r}$, the following constraints are to be considered:
\begin{align}
&\sum_{i=1}^{\omega}p_i - 1 = 0\label{cons1}\\
&\sum_{i=1}^{\omega}(p_i\epsilon_i - E) = 0.\label{cons2}
\end{align}
Now, we construct the three-parameter entropic functional, denoted by $\Phi_{p,q',r}$, by adding the above constraints \eqref{cons1} and \eqref{cons2} to the entropy $S_{p,q',r}$ with Lagrange multipliers. That is,
\begin{equation}
\Phi_{p,q',r}(p_i, \alpha,\beta)=\frac{1}{k}S_{p,q',r}+\alpha\left(\sum_{i=1}^{\omega}p_i - 1\right)+\beta\sum_{i=1}^{\omega}(p_i\epsilon_i - E).
\end{equation}
The entropic functional $\Phi_{p,q',r}$ should be maximized in order to reach the equilibrium state. Hence,
\begin{equation}\label{entro_func1}
\frac{\partial\Phi_{p,q',r}(p_i, \alpha,\beta)}{\partial p_i}=\frac{1}{k}\frac{\partial S_{p,q',r}}{\partial p_i}+\alpha+\beta\epsilon_i=0.
\end{equation}
Note that
\begin{equation}\label{entro_func2}
\frac{1}{k}\frac{\partial S_{p,q',r}}{\partial p_i}=p_i\frac{\partial \ln_{p,q',r}\frac{1}{p_i}}{\partial p_i}+\ln_{p,q',r}\frac{1}{p_i}
\end{equation}
with
\begin{align*}
\frac{\partial \ln_{p,q',r}\frac{1}{p_i}}{\partial p_i}&=\frac{1}{1-r}\exp\left(\frac{1-r}{1-q'}\left(\exp\left(\frac{1-q'}{1-q}\left(p_i^{q-1}-1\right)\right)-1\right)\right)\times\\
&\;\;\;\;\;\;\;\;\;\;\frac{1-r}{1-q'}\exp\left(\frac{1-q'}{1-q}\left(p_i^{q-1}-1\right)\right)\frac{1-q'}{1-q}(q-1)p_i^{q-2}\\
&=e^{-\frac{1}{1-r}}\exp\left(\frac{1-r}{1-q'}\exp\left(\frac{1-q'}{1-q}\left(p_i^{q-1}-1\right)\right)\right)\times\\
&\;\;\;\;\;\;\;\;\;\;\exp\left(\frac{1-q'}{1-q}\left(p_i^{q-1}-1\right)\right)(-p_i^{q-2}).
\end{align*}
Letting
\begin{equation}\label{eqn_u} 
u=\exp\left(\frac{1-q'}{1-q}\left(p_i^{q-1}-1\right)\right)
\end{equation} 
yields
\begin{align*}
\frac{\partial \ln_{p,q',r}\frac{1}{p_i}}{\partial p_i}=e^{-\frac{1-r}{1-q'}}e^{\frac{1-r}{1-q'}u}u(-p_i^{q-2}).
\end{align*}
Then
\begin{align*}
\frac{\partial\Phi_{p,q',r}(p_i, \alpha,\beta)}{\partial p_i}&=-p_i^{q-1}e^{-\frac{1-r}{1-q'}}e^{\frac{1-r}{1-q'}u}u+\frac{1}{1-r}\left[e^{-\frac{1-r}{1-q'}}e^{\frac{1-r}{1-q'}u}-1\right]\\
&\;\;\;\;\;\;\;\;\;\;+\alpha+\beta\epsilon_i=0.
\end{align*}
\begin{align*}
-p_i^{q-1}e^{-\frac{1-r}{1-q'}}e^{\frac{1-r}{1-q'}u}u+\frac{1}{1-r}e^{-\frac{1-r}{1-q'}}e^{\frac{1-r}{1-q'}u}-\frac{1}{1-r}+\alpha+\beta\epsilon_i=0.
\end{align*}
But equation \eqref{eqn_u} can be written as
\begin{align*}
\ln u&=\frac{1-q'}{1-q}\left(p_i^{q-1}-1\right)\\
p_i^{q-1}&= 1+\frac{1-q}{1-q'}\ln u.
\end{align*}
Hence,
\begin{align*}
&-\left(1+\frac{1-q}{1-q'}\ln u\right)e^{\frac{1-r}{1-q'}u}u+\frac{1}{1-r}e^{\frac{1-r}{1-q'}u}+\left(-\frac{1}{1-r}+\alpha+\beta\epsilon_i\right)e^{-\frac{1-r}{1-q'}}=0\\
&e^{\frac{1-r}{1-q'}u}u+\frac{1-q}{1-q'}u(\ln u) e^{\frac{1-r}{1-q'}u}-\frac{1}{1-r}e^{\frac{1-r}{1-q'}u}=\left(-\frac{1}{1-r}+\alpha+\beta\epsilon_i\right)e^{-\frac{1-r}{1-q'}}\\
&\;\;\;\;\;\;\;\;\;\;\frac{1-r}{1-q'}ue^{\frac{1-r}{1-q'}u}+\frac{1-q}{1-q'}\frac{1-r}{1-q'}u(\ln u) e^{\frac{1-r}{1-q'}u}-\frac{1-r}{1-q'}\frac{1}{1-r}e^{\frac{1-r}{1-q'}u}\\
&\;\;\;\;\;\;\;\;\;\;\;\;\;\;\;\;\;\;\;\;=\left(-\frac{1}{1-r}+\alpha+\beta\epsilon_i\right)\frac{1-r}{1-q'}e^{-\frac{1-r}{1-q'}}.
\end{align*}
By taking $y=\frac{1-r}{1-q'}u$, we obtain
\begin{equation*}
ye^y+\frac{1-q}{1-q'}y\ln\left(\frac{1-q'}{1-r}y\right) e^y-\frac{1}{1-q'}e^y=x
\end{equation*}
where
\begin{equation}\label{eqn_x}
x=\left(-\frac{1}{1-r}+\alpha+\beta\epsilon_i\right)\frac{1-r}{1-q'}e^{-\frac{1-r}{1-q'}}.
\end{equation}
Thus,
\begin{equation*}
\left(\frac{1-q}{1-q'}y\ln\left(\frac{1-q'}{1-r}y\right)+y-\frac{1}{1-q'}\right)e^y=x.
\end{equation*}
With 
\begin{equation}\label{constants}
A=\frac{1-q}{1-q'},  B=\frac{1-q'}{1-r}, C=-\frac{1}{1-q'},
\end{equation}
it follows that 
\begin{equation*}
\left(Ay\ln\left(By\right)+y+C\right)e^y=x.
\end{equation*}
This implies that
\begin{align*}
y&=W_{\mathcal{LT}}(x)\\
\frac{1-r}{1-q'}u&=W_{\mathcal{LT}}(x)\\
u&=\frac{1-q'}{1-r}W_{\mathcal{LT}}(x)\\
\end{align*}
Using equation \eqref{eqn_u}
\begin{align*}
&\exp\left(\frac{1-q'}{1-q}\left(p_i^{q-1}-1\right)\right)
=\frac{1-q'}{1-r}W_{\mathcal{LT}}(x)\\
&\frac{1-q'}{1-q}\left(p_i^{q-1}-1\right)
=\ln\left(\frac{1-q'}{1-r}W_{\mathcal{LT}}(x)\right).
\end{align*}
Therefore,  the probability distribution is given by
\begin{equation}\label{proba_distn}
p_i
=\frac{1}{Z_{q,q',r}}\left\{\frac{1-q}{1-q'}\ln\left(\frac{1-q'}{1-r}W_{\mathcal{LT}}(x)\right)+1\right\}^{\frac{1}{q-1}}
\end{equation}
where
$$Z_{q,q',r}=\sum_{i=1}^{\omega}\left\{\frac{1-q}{1-q'}\ln\left(\frac{1-q'}{1-r}W_{\mathcal{LT}}(x)\right)+1\right\}^{\frac{1}{q-1}}.$$
\begin{align*}
x&=\left(1-\alpha(1-r)-\beta(1-r)\epsilon_i\right)\frac{1}{q'-1}e^{-\frac{1-r}{1-q'}}\\
&=\frac{1}{q'-1}e^{-\frac{1-r}{1-q'}}(1-\alpha(1-r))\left(1-\frac{\beta(1-r)}{1-\alpha(1-r)}\epsilon_i\right)\\
&=\frac{1}{q'-1}e^{-\frac{1-r}{1-q'}}(1-\alpha(1-r))\left(1-\beta_r(1-r)\epsilon_i\right)\\
&=\frac{1}{q'-1}e^{-\frac{1-r}{1-q'}}(1-\alpha(1-r))\left[\exp_r(-\beta_r\epsilon_i)\right],
\end{align*}
where $\beta_r$ may be defined as the inverse of the pseudo-temperature
$$\beta_r\equiv\frac{1}{k_rT_r}=\frac{\beta}{1-\alpha(1-r)},$$
\begin{align*}
p_i
&=\frac{1}{Z_{q,q',r}}\left\{1+(1-q)\ln\left(\frac{1-q'}{1-r}W_{\mathcal{LT}}\left(\frac{e^{\frac{1-r}{q'-1}}(1-\alpha(1-r))e_r^{-\beta_r\epsilon_i}}{q'-1}\right)\right)^{\frac{1}{1-q'}}\right\}^{\frac{1}{q-1}}\\
&=\frac{1}{Z_{q,q',r}}\left\{\exp_q\left(\ln\left(\frac{1-q'}{1-r}W_{\mathcal{LT}}\left(\frac{e^{\frac{1-r}{q'-1}}(1-\alpha(1-r))e_r^{-\beta_r\epsilon_i}}{q'-1}\right)\right)^{\frac{1}{1-q'}}\right)\right\}^{-1}.
\end{align*}
We can assume the energy level, $\epsilon_i$, as a quadratic function of the variable $x_i$. The continuous normalized probability distribution of $x$ can then be rewritten as
\begin{align*}
p(x)=\frac{\left\{1+(1-q)\ln\left(\frac{1-q'}{1-r}W_{\mathcal{LT}}\left(\frac{e^{\frac{1-r}{q'-1}}(1-\alpha(1-r))e_r^{-\beta_rx^2}}{q'-1}\right)\right)^{\frac{1}{1-q'}}\right\}^{\frac{1}{q-1}}}{\int_{-\infty}^{\infty}\left\{1+(1-q)\ln\left(\frac{1-q'}{1-r}W_{\mathcal{LT}}\left(\frac{e^{\frac{1-r}{q'-1}}(1-\alpha(1-r))e_r^{-\beta_rx^2}}{q'-1}\right)\right)^{\frac{1}{1-q'}}\right\}^{\frac{1}{q-1}}dx}.
\end{align*}

\section{Conclusion}
In this paper, a special set of three-parameter entropies \cite{Corcino-Corcino} were maximized in the canonical ensemble by the energy constraint
$$\sum_{i=1}^{\omega}p_i\epsilon_i=E.$$
It is expected that the probability distribution, $p_i(\epsilon_i)$, can be expressed in terms of the generalized three-parameter exponential defined in \cite{Corcino-Corcino}. However, an interesting form of the solution of the related equation is obtained expressing the solution in terms of the translated logarithmic Lambert function which is a generalization of the classical Lambert W function.

\section*{Acknowledgment}
This research is funded by Cebu Normal University (CNU) and the Commission  on Higher Education - Grants-in-Aid (CHED-GIA) for Research.

\section*{Data Availability Statement}
The computer programs and articles used to generate the graphs and support the findings of this study are available from the corresponding author upon request.

\end{document}